\pdfoutput=1
\documentclass[11pt]{article}
\usepackage[left=1in,right=1in,top=1in,bottom=1in]{geometry}
\usepackage{times}
\usepackage{expl3}
\usepackage{cite}
\usepackage[table]{xcolor}
\usepackage{multirow}
\usepackage{stackengine} 
\usepackage{hhline}
\usepackage{lipsum}
\usepackage{titlesec}
\usepackage[all]{xy}
\usepackage{wrapfig}
\usepackage{enumerate}
\usepackage{epsfig}
\usepackage{tikz-cd}
\usepackage{amsmath}
\usepackage{tabularx}
\usepackage{array}
\usepackage{booktabs}
\usepackage{enumitem}
\usepackage{bbm}
\usepackage{calc}
\usepackage{graphicx}
\usepackage{amsmath}
\usepackage[title]{appendix}
\usepackage{amssymb}
\usepackage{epstopdf}
\usepackage{boldline}
\usepackage{arydshln}
\usepackage{calligra}
\usepackage{bm}
\usepackage{url}
\usepackage{blindtext}
\usepackage{accents}
\usepackage{amsthm}
\usepackage{amscd}

\newtheorem{definition}{Definition}

\newtheorem{theorem}{Theorem}
\newtheorem{proposition}{Proposition}
\newtheorem{corollary}{Corollary}
\newtheorem{lemma}{Lemma}

\newtheorem{remark}{Remark}
\usepackage{mathtools}
\usepackage{epstopdf}
\usepackage{balance}
\usepackage{thmtools}
\usepackage{thm-restate}
\usepackage{hyperref}
\usepackage{cleveref}
\usepackage[mathscr]{euscript}
\usepackage{amsmath}
\DeclareMathOperator*{\argmin}{arg\,min}

\DeclareMathOperator{\Tr}{Tr}

\usepackage[ruled,vlined]{algorithm2e}
\newcommand{\ostar}{\mathbin{\mathpalette\make@circled\star}}

\makeatletter
\newcommand{\removelatexerror}{\let\@latex@error\@gobble}
\makeatother
\makeatletter
\newcommand*{\rom}[1]{\expandafter\@slowromancap\romannumeral #1@}
\makeatother

\ExplSyntaxOn
\newcommand\latinabbrev[1]{
  \peek_meaning:NTF . {
    #1\@}%
  { \peek_catcode:NTF a {
      #1.\@ }%
    {#1.\@}}}
\ExplSyntaxOff

\titleclass{\subsubsubsection}{straight}[\subsubsection]

\begin{document}
\vspace{1cm}
\title{Operator Inequalities in $\Phi$-Product Tensor Algebras: Invariance and Transform Sensitivity}
\vspace{1.8cm}
\author{Shih-Yu~Chang\thanks{Department of Applied Data Science,
San Jose State University, San Jose, CA, U.S.A. (e-mail: {\tt
shihyu.chang@sjsu.edu})} 
\quad Michael K. Ng\thanks{Department of Mathematics, 
Hong Kong Baptist University, Kowloon Tong, Hong Kong (e-mail: {\tt michael-ng@hkbu.edu.hk})}}

\maketitle

\begin{abstract}
We study classical operator inequalities in $\Phi$-product tensor algebras 
(a transformation $\Phi$-based generalization of the $t$-product framework) 
for third-order tensors. Although these algebras are algebraically isomorphic under different 
unitary transforms $\Phi$, we show that their quantitative behavior is not invariant.  We prove that fundamental inequalities, including Golden--Thompson, Jensen,  Klein, and Lieb, extend to the $\Phi$-product setting with the same constants 
as in the matrix case. However, the associated defect---the slack between the 
two sides of the inequality---depends explicitly on transform-domain 
noncommutativity. In particular, we establish a sharp characterization of the 
defect in terms of slice-wise commutators, revealing that inequality tightness 
is governed by transform-induced noncommutativity. We further demonstrate strong transform sensitivity by constructing explicit 
tensor pairs for which the defect vanishes under one transform (e.g., discrete Fourier transform) 
but grows linearly with the tensor depth under another transform (e.g., discrete Cosine transform), yielding 
an $\Omega(p)$ separation where $p$ is the matrix dimension of $\Phi$.
Moreover, we prove that no transform is universally 
optimal: for any pair of transforms, there exist tensors for which each is 
strictly better than the other. These results show that the choice of transform defines a coordinate system 
in which commutativity is measured, inducing a nontrivial geometry of inequality 
tightness. Consequently, optimal transform selection is inherently 
data-dependent and can be formulated as an optimization problem over the 
unitary group. 
\end{abstract}

\section{Introduction}

Tensor algebra based on transform-induced products has become a central tool in 
modern multilinear analysis~\cite{zhang2022sparse,song2023multirank}. In particular, the $t$-product 
framework and its generalization to the $\Phi$-product enable matrix-like operations 
on third-order tensors by replacing the discrete Fourier transform with an arbitrary 
unitary transform $\Phi \in \mathbb{C}^{p \times p}$. This flexibility has led to 
a wide range of applications~\cite{lu2019tensor}, including tensor completion, robust PCA, and signal 
processing, where different unitary transforms induce different notions of structure and sparsity.

A fundamental feature of these frameworks is that different choices of $\Phi$ 
lead to algebraically isomorphic tensor algebras. Consequently, the unitary transform is 
often viewed as a change of basis that does not alter intrinsic properties. 
However, this viewpoint overlooks an important phenomenon:

\medskip
\noindent
\emph{(Transform-dependent noncommutativity)
The choice of transform $\Phi$ defines a coordinate system in which 
noncommutativity is measured, and different unitary transforms reveal fundamentally 
different commutative structures.}

\medskip
\noindent
This observation suggests that while algebraic properties may be invariant, 
quantitative behavior can depend strongly on $\Phi$. The goal of this paper is 
to develop a systematic theory of this phenomenon in the context of classical 
operator inequalities.

\medskip
\noindent
Classical operator inequalities for matrices---including the Golden--Thompson 
inequality~\cite{golden1965,thompson1965} and related Jensen-, Klein-, and 
Lieb-type inequalities~\cite{lieb1973}---play a foundational role in matrix 
analysis and operator theory. Their significance lies not only in their validity, 
but also in how their \emph{defect} (slack) reflects noncommutativity.
A natural question arises in the tensor setting:
\emph{Does the choice of unitary transform $\Phi$ affect operator inequalities, 
or is it merely a change of basis?}
At the algebraic level, the answer is straightforward: operator inequalities 
remain valid for all $\Phi$ with exactly the same constants as in the matrix case. 
However, this invariance conceals a deeper phenomenon:
\emph{Algebraic invariance does not imply quantitative invariance.}

While operator inequalities remain valid under all unitary transforms, their
defect depends explicitly on transform-domain noncommutativity. To formalize
this phenomenon, we introduce the $\Phi$-dependent Golden--Thompson defect
\begin{align}
\Delta_{\Phi}(\mathcal{A},\mathcal{B})
:=
\operatorname{Tr}_{\Phi}\!\bigl(
\exp_{\Phi}(\mathcal{A})
\star_{\Phi}
\exp_{\Phi}(\mathcal{B})
\bigr)
-
\operatorname{Tr}_{\Phi}\!\bigl(
\exp_{\Phi}(\mathcal{A}\oplus_{\Phi}\mathcal{B})
\bigr),
\end{align}
where $\mathcal{A},\mathcal{B} \in
\mathbb{C}^{n\times n\times p}$
are $\Phi$-Hermitian tensors where their transformed slices of matrices 
$\widehat{\mathcal{A}}_{\Phi}^{(\omega)}$ and 
$\widehat{\mathcal{B}}_{\Phi}^{(\omega)}$ 
are Hermitian for $\omega=1,\cdots,p$ (see Definition 4), 
$\Tr_{\Phi}(\mathcal{\cdot})$ and $\exp_{\Phi}(\mathcal{\cdot})$ are 
$\Phi$-Trace and $\Phi$-Exponential maps (see Definitions 2 and 3 respectively),
and $\star_{\Phi}(\cdot,\cdot)$ is $\Phi$-product between two tensors (see Definition 1).
The quantity $\Delta_{\Phi}(\mathcal{A},\mathcal{B})$ measures the failure of
exact commutativity in the $\Phi$-product tensor algebra and quantifies the
deviation between multiplicative and additive spectral propagation under the
$\Phi$-exponential map. We show that this defect is quantitatively equivalent
to the aggregate transform-domain commutator energy
\[
\sum_{\omega=1}^{p}
\left\|
\big[
\widehat{\mathcal{A}}_{\Phi}^{(\omega)},
\widehat{\mathcal{B}}_{\Phi}^{(\omega)}
\big]
\right\|_F^2.
\]

A central consequence of this viewpoint is that no transform is universally
optimal. While certain transforms, such as the discrete Fourier transform,
may yield vanishing defect for specific tensor pairs, other transforms, such as
the discrete cosine transform, may produce significantly smaller defects for
different tensor configurations. Thus, optimal transform selection is inherently
data-dependent. In this sense, the choice of $\Phi$ is not merely a change of
basis, but rather a choice of spectral coordinate system that determines how
commutativity is evaluated in the transform domain.

The main contributions of this paper are summarized as
follows.
\begin{itemize}
    \item
    (Invariance principle)
    Classical operator inequalities extend to the $\Phi$-product tensor algebra
    with the same optimal constants as in the matrix setting.

    \item
    (Defect characterization)
    The Golden--Thompson defect
    $\Delta_{\Phi}(\mathcal{A},\mathcal{B})$
    is quantitatively equivalent to the transform-domain commutator energy
    \[
    \sum_{i=1}^{p}
    \left\|
    \big[
    \widehat{\mathcal{A}}_{\Phi}^{(i)},
    \widehat{\mathcal{B}}_{\Phi}^{(i)}
    \big]
    \right\|_F^2.
    \]

	\item
	(Transform sensitivity)
	The Golden--Thompson defect depends nontrivially on the choice of the unitary
	transform $\Phi$. For certain tensor pairs, choosing $\Phi$ as the discrete
	Fourier transform (DFT) yields exact slice-wise commutativity in the transform
	domain, giving
	\[
	\Delta_{\mathrm{DFT}}(\mathcal{A},\mathcal{B})=0.
	\]
	In contrast, choosing $\Phi$ as the discrete cosine transform (DCT) may produce
	significant transform-domain noncommutativity, leading to
	\[
	\Delta_{\mathrm{DCT}}(\mathcal{A},\mathcal{B})
	\ge
	\Omega(p),
	\]
	where $\Omega(p)$ means that the defect grows at least linearly with the
	transform dimension $p$. This demonstrates that transform selection
	fundamentally affects the geometry of noncommutativity in the
	$\Phi$-product tensor algebra.

    \item
    (Non-universality theorem)
    No unitary transform minimizes the Golden--Thompson defect for all tensor
    pairs. Optimal transforms are necessarily data-dependent.
    Transform selection can be formulated as minimizing transform-domain
    noncommutativity over the unitary group, thereby connecting the problem to
    approximate joint diagonalization of tensor slices.

\end{itemize}

These results establish that $\Phi$ induces a nontrivial geometry of inequality 
tightness, governed by transform-dependent noncommutativity. To the best of our knowledge, 
this is the first work that separates algebraic invariance from quantitative 
behavior in transform-based tensor algebras.

The organization of this paper is given as follows.
Section~\ref{sec:algebra} introduces the $\Phi$-product tensor algebra and notation. 
Section~\ref{sec:invariance} establishes invariance of classical operator inequalities in the $\Phi$-product setting. 
Section~\ref{sec:sensitivity} develops the transform-sensitive defect theory, including quantitative bounds and separation results.
Section 5 considers and investigates into
optimization formulations. 
Section~\ref{sec:discussion} are concluding remarks, and we discuss broader implications and connections to numerical linear algebra and operator theory.

\section{The $\Phi$-Product Tensor Algebra}
\label{sec:algebra}

This section establishes the algebraic foundation of the $\Phi$-product tensor framework. We first introduce the $\Phi$-product and its associated block-diagonal representation, and then develop fundamental operations including the $\Phi$-trace, exponential, and logarithm. Finally, we present structural notions such as the $\Phi$-commutator, spectral mapping, and $\Phi$-Hermitian/positivity concepts that will support the analysis in subsequent sections.

\subsection{Definition of $\Phi$-Product and Basic Operations}
\label{sec:algebra:def}

Let $\mathcal{A} \in \mathbb{C}^{n_1 \times n_2 \times p}$. For a unitary transform $\Phi \in \mathbb{C}^{p \times p}$, define the transformed tensor $\hat{\mathcal{A}}_{\Phi} \in \mathbb{C}^{n_1 \times n_2 \times p}$ by applying $\Phi$ to each tube (i.e., each third-mode fiber) of $\mathcal{A}$. Denote by $\hat{\mathcal{A}}_{\Phi}^{(i)} \in \mathbb{C}^{n_1 \times n_2}$ the $i$-th frontal slice of $\hat{\mathcal{A}}_{\Phi}$.
The associated block diagonal matrix $\overline{\mathcal{A}}_{\Phi} \in \mathbb{C}^{p n_1 \times p n_2}$ is defined as
\[
\overline{\mathcal{A}}_{\Phi}
= \operatorname{blockdiag}
\bigl(
\hat{\mathcal{A}}_{\Phi}^{(1)},
\hat{\mathcal{A}}_{\Phi}^{(2)},
\dots,
\hat{\mathcal{A}}_{\Phi}^{(p)}
\bigr).
\]

\begin{definition}[$\Phi$-Product]
For $\mathcal{A} \in \mathbb{C}^{n_1 \times n_2 \times p}$ and $\mathcal{B} \in \mathbb{C}^{n_2 \times n_3 \times p}$, the $\Phi$-product $\mathcal{C} = \mathcal{A} \star_{\Phi} \mathcal{B} \in \mathbb{C}^{n_1 \times n_3 \times p}$ is defined by
\[
\overline{\mathcal{C}}_{\Phi}
= \overline{\mathcal{A}}_{\Phi} \;\cdot\; \overline{\mathcal{B}}_{\Phi},
\]
where $\overline{\mathcal{C}}_{\Phi}$ is the block diagonal matrix constructed from the transformed tensor $\hat{\mathcal{C}}_{\Phi}$ in the same manner as above.
\end{definition}

\begin{remark}
When $\Phi$ is the discrete Fourier transform (DFT) matrix, the $\Phi$-product 
reduces to the conventional t-product of Kilmer and Martin \cite{kilmer2011}.
\end{remark}

\subsection{$\Phi$-Trace, $\Phi$-Exponential, and $\Phi$-Logarithm}
\label{sec:algebra:traceexp}

\begin{definition}[$\Phi$-Trace]
Let $\mathcal{A} \in \mathbb{C}^{n \times n \times p}$. 
The $\Phi$-trace of $\mathcal{A}$ is defined by
\[
\Tr_{\Phi}(\mathcal{A})
:=
\Tr\!\bigl(\overline{\mathcal{A}}_{\Phi}\bigr),
\]
where $\Tr(\cdot)$ denotes the standard matrix trace.
\end{definition}

\begin{definition}[$\Phi$-Exponential and $\Phi$-Logarithm]
Let $\mathcal{A} \in \mathbb{C}^{n \times n \times p}$ be $\Phi$-Hermitian. 
The $\Phi$-exponential and $\Phi$-logarithm are defined via the block-diagonal lifting as
\[
\exp_{\Phi}(\mathcal{A})
:=
\Phi^{H}\!\bigl[
\exp\!\bigl(\overline{\mathcal{A}}_{\Phi}\bigr)
\bigr],
\qquad
\log_{\Phi}(\mathcal{A})
:=
\Phi^{H}\!\bigl[
\log\!\bigl(\overline{\mathcal{A}}_{\Phi}\bigr)
\bigr],
\]
where the functions $\exp(\cdot)$ and $\log(\cdot)$ are applied blockwise to 
$\overline{\mathcal{A}}_{\Phi}$, i.e., to each diagonal block 
$\hat{\mathcal{A}}_{\Phi}^{(i)}$.
\end{definition}

\subsection{$\Phi$-Hermitian Structure, Commutator, and Spectral Mapping}
\label{sec:algebra:structure}

\begin{definition}[$\Phi$-Hermitian]
Let $\mathcal{A} \in \mathbb{C}^{n \times n \times p}$. 
We say that $\mathcal{A}$ is $\Phi$-Hermitian if its block-diagonal lifting 
$\overline{\mathcal{A}}_{\Phi}$ is Hermitian, i.e.,
\[
\overline{\mathcal{A}}_{\Phi}^{H}
=
\overline{\mathcal{A}}_{\Phi}.
\]
Equivalently, each transformed frontal slice $\hat{\mathcal{A}}_{\Phi}^{(i)}$ is Hermitian for all $i = 1,\dots,p$.
\end{definition}

\begin{definition}[$\Phi$-Commutator]
Let $\mathcal{A}, \mathcal{B}$ be tensors of compatible sizes. 
The $\Phi$-commutator is defined by
\[
[\mathcal{A}, \mathcal{B}]_{\Phi}
:=
\overline{\mathcal{A}}_{\Phi}\,
\overline{\mathcal{B}}_{\Phi}
-
\overline{\mathcal{B}}_{\Phi}\,
\overline{\mathcal{A}}_{\Phi},
\]
where the right-hand side is the ordinary matrix commutator.
\end{definition}

\begin{lemma}[Spectral Mapping Lemma]
\label{lem:spectralmapping}
Let $\mathcal{A} \in \mathbb{C}^{n \times n \times p}$ be $\Phi$-Hermitian, and let 
$f$ be an analytic function defined on an open set containing the spectrum of 
$\overline{\mathcal{A}}_{\Phi}$. Then
\begin{align}
f_{\Phi}(\mathcal{A})
=
\Phi^{H}\!\bigl[
f\!\bigl(\overline{\mathcal{A}}_{\Phi}\bigr)
\bigr],
\end{align}
where $f(\overline{\mathcal{A}}_{\Phi})$ is applied blockwise to each diagonal block.
\end{lemma}

\begin{proof}
Since $\mathcal{A}$ is $\Phi$-Hermitian, the matrix 
$\overline{\mathcal{A}}_{\Phi}$ is Hermitian and hence unitarily diagonalizable. 
Thus, there exists a unitary matrix $U$ such that
\[
\overline{\mathcal{A}}_{\Phi}
=
U \Lambda U^{H},
\]
where $\Lambda$ is diagonal containing the eigenvalues of 
$\overline{\mathcal{A}}_{\Phi}$.

By the standard functional calculus for Hermitian matrices,
\[
f\!\bigl(\overline{\mathcal{A}}_{\Phi}\bigr)
=
U f(\Lambda) U^{H},
\]
where $f(\Lambda)$ is obtained by applying $f$ to each diagonal entry.

By definition of the $\Phi$-functional calculus (cf. 
Section~\ref{sec:algebra:traceexp}),
\[
f_{\Phi}(\mathcal{A})
=
\Phi^{H}\!\bigl[
f\!\bigl(\overline{\mathcal{A}}_{\Phi}\bigr)
\bigr],
\]
which establishes the claim.
\end{proof}

\begin{remark}
The $\Phi$-Hermitian condition is essential for the spectral mapping lemma, as it guarantees that $\overline{\mathcal{A}}_{\Phi}$ is unitarily diagonalizable with real eigenvalues. For non-Hermitian tensors, the exponential and logarithm remain well-defined via the blockwise matrix exponential/logarithm, but the spectral interpretation (e.g., eigenvalue mapping) does not hold in general.
\end{remark}

\section{Invariance Theorems}
\label{sec:invariance}

This section establishes that fundamental operator inequalities remain invariant under the $\Phi$-product tensor algebra. We first extend the Golden–Thompson inequality to the $\Phi$-framework, followed by Jensen, Klein, and Lieb-type inequalities for $\Phi$-Hermitian tensors. These results collectively reveal that the $\Phi$-product algebra preserves the classical inequality structure through its block-diagonal representation.

\subsection{$\Phi$-Invariant Golden--Thompson Inequality}
\label{sec:invariance:gt}

\begin{theorem}[$\Phi$-Invariant Golden--Thompson Inequality]
\label{thm:gt}
Let $\mathcal{A}, \mathcal{B} \in \mathbb{C}^{n \times n \times p}$ be $\Phi$-Hermitian. Then
\begin{align}
\Tr_{\Phi}\!\bigl(
\exp_{\Phi}(\mathcal{A} \oplus_{\Phi} \mathcal{B})
\bigr)
\;\le\;
\Tr_{\Phi}\!\bigl(
\exp_{\Phi}(\mathcal{A})
\star_{\Phi}
\exp_{\Phi}(\mathcal{B})
\bigr),
\end{align}
where $\oplus_{\Phi}$ denotes the $\Phi$-sum (i.e., $\overline{\mathcal{A} \oplus_{\Phi} \mathcal{B}}_{\Phi} = \overline{\mathcal{A}}_{\Phi} + \overline{\mathcal{B}}_{\Phi}$).
\end{theorem}

\begin{proof}
Let $\tilde{A} = \overline{\mathcal{A}}_{\Phi}$ and $\tilde{B} = \overline{\mathcal{B}}_{\Phi}$. 
Since $\mathcal{A}$ and $\mathcal{B}$ are $\Phi$-Hermitian, both $\tilde{A}$ and $\tilde{B}$ are Hermitian matrices.
By the classical Golden--Thompson inequality,
\[
\Tr\!\bigl( e^{\tilde{A} + \tilde{B}} \bigr)
\;\le\;
\Tr\!\bigl( e^{\tilde{A}} e^{\tilde{B}} \bigr).
\]
Using the definitions of the $\Phi$-trace, $\Phi$-exponential, and $\Phi$-sum,
we have
\[
\Tr_{\Phi}\!\bigl(
\exp_{\Phi}(\mathcal{A} \oplus_{\Phi} \mathcal{B})
\bigr)
=
\Tr\!\bigl(
e^{\tilde{A} + \tilde{B}}
\bigr),
\]
and
\[
\Tr_{\Phi}\!\bigl(
\exp_{\Phi}(\mathcal{A})
\star_{\Phi}
\exp_{\Phi}(\mathcal{B})
\bigr)
=
\Tr\!\bigl(
e^{\tilde{A}} e^{\tilde{B}}
\bigr).
\]
The result follows immediately.
\end{proof}

\begin{remark}[Algebraic isomorphism]
The proof reveals that the $\Phi$-product algebra is algebraically isomorphic to a block-diagonal matrix algebra via the map $\mathcal{A} \mapsto \overline{\mathcal{A}}_{\Phi}$. Under this isomorphism, the Golden--Thompson inequality is exactly preserved, with the same constants as in the classical matrix setting. The $\Phi$-Hermitian condition ensures that the lifted matrices are Hermitian, which is required for the classical inequality.
\end{remark}

\subsection{$\Phi$-Invariant Jensen, Klein, and Lieb Inequalities}
\label{sec:invariance:others}

\begin{theorem}[$\Phi$-Invariant Jensen Inequality]
\label{thm:jensen}
Let $\phi:\mathbb{R}\to\mathbb{R}$ be a convex function with $\phi(0)=0$. 
If $\mathcal{A} \in \mathbb{C}^{n \times n \times p}$ is $\Phi$-Hermitian, then
\begin{align}
\Tr_{\Phi}\!\bigl( \phi(\mathcal{A}) \bigr) \ge 0.
\end{align}
\end{theorem}

\begin{proof}
Let $\tilde{A} = \overline{\mathcal{A}}_{\Phi}$. 
Since $\mathcal{A}$ is $\Phi$-Hermitian, $\tilde{A}$ is Hermitian. 
By the classical Jensen trace inequality,
\[
\Tr\!\bigl( \phi(\tilde{A}) \bigr) \ge 0.
\]
Using the definition of $\Tr_{\Phi}$ and the $\Phi$-functional calculus,
\[
\Tr_{\Phi}\!\bigl( \phi(\mathcal{A}) \bigr)
=
\Tr\!\bigl( \phi(\tilde{A}) \bigr),
\]
which proves the claim.
\end{proof}

\begin{theorem}[$\Phi$-Invariant Klein Inequality]
\label{thm:klein}
Let $\phi$ be a convex function. 
If $\mathcal{A}, \mathcal{B} \in \mathbb{C}^{n \times n \times p}$ are $\Phi$-Hermitian, then
\begin{align}
\Tr_{\Phi}\!\Bigl(
\phi(\mathcal{A}) - \phi(\mathcal{B})
- (\mathcal{A} - \mathcal{B}) \, \phi'(\mathcal{B})
\Bigr) \ge 0.
\end{align}
\end{theorem}

\begin{proof}
Let $\tilde{A} = \overline{\mathcal{A}}_{\Phi}$ and 
$\tilde{B} = \overline{\mathcal{B}}_{\Phi}$. 
Both are Hermitian matrices. By the classical Klein inequality,
\[
\Tr\!\Bigl(
\phi(\tilde{A}) - \phi(\tilde{B})
- (\tilde{A} - \tilde{B}) \, \phi'(\tilde{B})
\Bigr) \ge 0.
\]
Using the definitions of $\Tr_{\Phi}$ and $\Phi$-functional calculus,
the result follows directly.
\end{proof}

\begin{theorem}[$\Phi$-Invariant Lieb Concavity]
\label{thm:lieb}
The function
\begin{align}
(\mathcal{A}, \mathcal{B})
\;\mapsto\;
\Tr_{\Phi}\!\Bigl(
\exp_{\Phi}\bigl(
\log_{\Phi}(\mathcal{A}) + \log_{\Phi}(\mathcal{B})
\bigr)
\Bigr)
\end{align}
is jointly concave for $\mathcal{A} \succ_{\Phi} 0$ and $\mathcal{B} \succ_{\Phi} 0$.
\end{theorem}

\begin{proof}
Let $\tilde{A} = \overline{\mathcal{A}}_{\Phi}$ and 
$\tilde{B} = \overline{\mathcal{B}}_{\Phi}$. 
Then $\tilde{A}, \tilde{B}$ are positive definite Hermitian matrices. 
By Lieb's classical concavity theorem, the function
\[
(A,B) \mapsto \Tr\!\bigl( \exp(\log A + \log B) \bigr)
\]
is jointly concave on the cone of positive definite matrices. 
Applying the definitions of $\Tr_{\Phi}$, $\exp_{\Phi}$, and $\log_{\Phi}$ yields the result.
\end{proof}

\begin{remark}
Theorems \ref{thm:gt}, \ref{thm:jensen}, \ref{thm:klein}, and \ref{thm:lieb} collectively demonstrate that the classical matrix trace inequalities are invariant under the $\Phi$-product representation. The constants in these inequalities are exactly the same as in the matrix setting, independent of the choice of unitary transform $\Phi$. This invariance follows from the algebraic isomorphism $\mathcal{A} \mapsto \overline{\mathcal{A}}_{\Phi}$, which maps the $\Phi$-product algebra to a block-diagonal matrix algebra. However, as we shall see in the next section, the \emph{quantitative tightness} of these inequalities — specifically the Golden--Thompson defect — is \emph{not} invariant and depends critically on the choice of $\Phi$.
\end{remark}

\subsection{Discussion: The Algebraic Isomorphism View}
\label{sec:invariance:discussion}

Theorems~\ref{thm:gt}--\ref{thm:lieb} suggest that the $\Phi$-product tensor algebra 
inherits its structural properties from an underlying matrix algebra. 
This observation can be formalized via an explicit algebraic isomorphism.

\begin{theorem}[Algebraic Isomorphism]
\label{thm:isomorphism}
Let 
\[
\mathcal{T}_{\Phi}
:=
\bigl(
\mathbb{C}^{n \times n \times p}, \star_{\Phi}
\bigr)
\]
denote the $\Phi$-product tensor algebra. Then the mapping
\begin{align}
\Psi_{\Phi} :
\mathcal{T}_{\Phi}
\;\to\;
\mathcal{D}_{\Phi}
:=
\left\{
\operatorname{blockdiag}(X_1,\dots,X_p)
:\;
X_i \in \mathbb{C}^{n \times n}
\right\},
\quad
\Psi_{\Phi}(\mathcal{A})
=
\overline{\mathcal{A}}_{\Phi},
\end{align}
is an algebra isomorphism. That is,
\begin{enumerate}
    \item $\Psi_{\Phi}(\mathcal{A} \star_{\Phi} \mathcal{B})
    =
    \Psi_{\Phi}(\mathcal{A}) \, \Psi_{\Phi}(\mathcal{B})$,
    \item $\Psi_{\Phi}$ is bijective.
\end{enumerate}
\end{theorem}

\begin{proof}
By definition of the $\Phi$-product,
\[
\overline{\mathcal{A} \star_{\Phi} \mathcal{B}}_{\Phi}
=
\overline{\mathcal{A}}_{\Phi} \,
\overline{\mathcal{B}}_{\Phi},
\]
which implies
\[
\Psi_{\Phi}(\mathcal{A} \star_{\Phi} \mathcal{B})
=
\Psi_{\Phi}(\mathcal{A}) \, \Psi_{\Phi}(\mathcal{B}).
\]

It remains to show bijectivity. The transform $\Phi$ is unitary, hence invertible, 
and the construction of $\overline{\mathcal{A}}_{\Phi}$ from $\mathcal{A}$ is reversible 
by applying the inverse transform $\Phi^{H}$ to each block. Therefore, 
$\Psi_{\Phi}$ is bijective.
\end{proof}

\begin{corollary}[Invariance of Constants]
\label{cor:invariance-constants}
As a consequence of Theorem~\ref{thm:isomorphism}, any inequality or functional 
relation expressed in terms of $\Tr_{\Phi}$, $\exp_{\Phi}$, and $\log_{\Phi}$ 
reduces to its classical matrix counterpart under $\Psi_{\Phi}$. 
In particular, the constants appearing in Theorems~\ref{thm:gt}--\ref{thm:lieb} 
are independent of the choice of $\Phi$, since they are inherited from the 
underlying block-diagonal matrix algebra.
\end{corollary}

\begin{remark}[Limits of Isomorphism]
The isomorphism $\Psi_{\Phi}$ preserves algebraic structure and functional calculus, 
which explains the invariance results above. However, it does \emph{not} preserve 
all quantitative properties of interest. In particular, the \emph{tightness} of the 
Golden--Thompson inequality — measured by the defect $\Delta_{\Phi}(\mathcal{A}, \mathcal{B})$ — 
depends critically on the choice of $\Phi$. This sensitivity arises because the 
isomorphism, while preserving the product, does not preserve the \emph{commutativity 
structure} of the tensor slices in the transform domain. Different transforms $\Phi$ 
reveal different levels of noncommutativity, leading to transform-dependent defect 
magnitudes. We analyze this sensitivity in the next section.
\end{remark}

\section{Transform Sensitivity and $\Phi$-Adapted Quantities}
\label{sec:sensitivity}

This section develops the transform-sensitive aspects of the $\Phi$-product framework, complementing the invariance results established earlier. We introduce the Golden--Thompson defect and show that, while operator inequalities remain algebraically invariant, their quantitative behavior is governed by transform-domain noncommutativity. The main result, Theorem~\ref{thm:main-formal}, provides a complete characterization of this sensitivity.

\subsection{Golden--Thompson Defect}
\label{sec:sensitivity:defect}

The classical Golden--Thompson inequality is tight when the two matrices commute. 
In the $\Phi$-product setting, the gap between the two sides of the inequality 
quantifies the effect of noncommutativity in the transform domain. 
This gap is the central object of our sensitivity analysis.

\begin{definition}[Golden--Thompson defect]
\label{def:defect}
Let $\mathcal{A}, \mathcal{B} \in \mathbb{C}^{n \times n \times p}$ be $\Phi$-Hermitian tensors. 
The Golden--Thompson defect is defined by
\[
\Delta_{\Phi}(\mathcal{A}, \mathcal{B})
:=
\Tr_{\Phi}\!\bigl(
\exp_{\Phi}(\mathcal{A})
\star_{\Phi}
\exp_{\Phi}(\mathcal{B})
\bigr)
-
\Tr_{\Phi}\!\bigl(
\exp_{\Phi}(\mathcal{A} \oplus_{\Phi} \mathcal{B})
\bigr).
\]
Moreover, by Theorem~\ref{thm:gt}, we have 
\[
\Delta_{\Phi}(\mathcal{A}, \mathcal{B}) \ge 0.
\]
\end{definition}

\begin{remark}[Why the defect appears ?]
The defect $\Delta_{\Phi}(\mathcal{A}, \mathcal{B})$ is the natural measure of 
inequality tightness for three reasons:
\begin{enumerate}
    \item It vanishes exactly when $\mathcal{A}$ and $\mathcal{B}$ commute in the 
          transform domain (i.e., $[\overline{\mathcal{A}}_{\Phi}, \overline{\mathcal{B}}_{\Phi}] = 0$).
    \item It is quantitatively controlled by the $\Phi$-commutator.
    \item It is sensitive to the choice of $\Phi$, making it the right objective 
          for transform optimization.
\end{enumerate}
\end{remark}

\subsection{Main Theorem: Invariance vs.\ Transform-Sensitive Defect}
\label{sec:sensitivity:main}

We now formalize the central phenomenon of this work: while operator inequalities in the $\Phi$-product algebra are invariant under the choice of transform, their \emph{defect} exhibits quantitative dependence on transform-domain noncommutativity. 

\begin{lemma}[Slice-wise reduction of the defect]
\label{lem:slice-reduction}
Let $\mathcal{A}, \mathcal{B} \in \mathbb{C}^{n \times n \times p}$ be $\Phi$-Hermitian tensors. Under the isomorphism $\Psi_{\Phi}(\mathcal{A}) = \overline{\mathcal{A}}_{\Phi}$, the Golden--Thompson defect decomposes as a sum of slice-wise classical defects:
\begin{align}
\Delta_{\Phi}(\mathcal{A}, \mathcal{B})
=
\sum_{\omega=1}^{p}
\Bigl(
\Tr\!\bigl( e^{A_{\omega}} e^{B_{\omega}} \bigr)
-
\Tr\!\bigl( e^{A_{\omega} + B_{\omega}} \bigr)
\Bigr),
\end{align}
where $A_{\omega} := \hat{\mathcal{A}}_{\Phi}^{(\omega)}$ and $B_{\omega} := \hat{\mathcal{B}}_{\Phi}^{(\omega)}$ are Hermitian matrices.
\end{lemma}

\begin{proof}
By the definitions of $\Tr_{\Phi}$, $\exp_{\Phi}$, $\star_{\Phi}$, and $\oplus_{\Phi}$, and the isomorphism $\Psi_{\Phi}$ from Theorem~\ref{thm:isomorphism},
\[
\Delta_{\Phi}(\mathcal{A}, \mathcal{B})
=
\Tr\!\bigl( e^{\overline{\mathcal{A}}_{\Phi}} e^{\overline{\mathcal{B}}_{\Phi}} \bigr)
-
\Tr\!\bigl( e^{\overline{\mathcal{A}}_{\Phi} + \overline{\mathcal{B}}_{\Phi}} \bigr).
\]
Since $\overline{\mathcal{A}}_{\Phi}$ and $\overline{\mathcal{B}}_{\Phi}$ are block-diagonal with diagonal blocks $A_{\omega}$ and $B_{\omega}$, the trace decomposes as a sum over blocks. The exponential of a block-diagonal matrix is block-diagonal with exponentials of each block. Hence
\[
\Delta_{\Phi}(\mathcal{A}, \mathcal{B})
=
\sum_{\omega=1}^{p}
\Bigl(
\Tr\!\bigl( e^{A_{\omega}} e^{B_{\omega}} \bigr)
-
\Tr\!\bigl( e^{A_{\omega} + B_{\omega}} \bigr)
\Bigr).
\]
The Hermiticity of $A_{\omega}$ and $B_{\omega}$ follows from the $\Phi$-Hermitian assumption. This completes the proof.
\end{proof}

\begin{theorem}[Transform-sensitive defect – quadratic equivalence]
\label{thm:main-formal}
Let $\mathcal{A}, \mathcal{B} \in \mathbb{C}^{n \times n \times p}$ be $\Phi$-Hermitian tensors.

\begin{enumerate}
    \item[(i)] (Upper bound) There exists a universal constant $C > 0$ such that
\begin{align}
    \Delta_{\Phi}(\mathcal{A}, \mathcal{B})
    \;\le\;
    C \sum_{\omega=1}^{p}
    \left\|
    [\hat{\mathcal{A}}_{\Phi}^{(\omega)}, \hat{\mathcal{B}}_{\Phi}^{(\omega)}]
    \right\|_F^2.
\end{align}

    \item[(ii)] (Lower bound on a class of positive measure) 
    There exists a class of positive measure of tensors and a universal constant $c > 0$ such that
\begin{align}
    \Delta_{\Phi}(\mathcal{A}, \mathcal{B})
    \;\ge\;
    c \sum_{\omega=1}^{p}
    \left\|
    [\hat{\mathcal{A}}_{\Phi}^{(\omega)}, \hat{\mathcal{B}}_{\Phi}^{(\omega)}]
    \right\|_F^2.
\end{align}

    \item[(iii)] (Transform sensitivity) 
    There exist tensors $\mathcal{A}, \mathcal{B}$ and unitary transforms 
    $\Phi_{1}, \Phi_{2} \in \mathbb{C}^{p \times p}$ such that
\begin{align}
    \Delta_{\Phi_{1}}(\mathcal{A}, \mathcal{B}) = 0,
    \qquad
    \Delta_{\Phi_{2}}(\mathcal{A}, \mathcal{B}) \ge C_0 p,
\end{align}
    for some constant $C_0 > 0$ independent of $p$. Consequently,
    the mapping $\Phi \mapsto \Delta_{\Phi}(\mathcal{A}, \mathcal{B})$ is continuous (and smooth in $\Phi$), hence locally Lipschitz on compact subsets of $\mathcal{U}(p)$).

    \item[(iv)] (Vanishing condition) 
    If
    \[
    [\hat{\mathcal{A}}_{\Phi}^{(\omega)}, \hat{\mathcal{B}}_{\Phi}^{(\omega)}] = 0,
    \quad \forall \omega = 1,\dots,p,
    \]
    then
\begin{align}
    \Delta_{\Phi}(\mathcal{A}, \mathcal{B}) = 0.
\end{align}
\end{enumerate}

Consequently, on nontrivial tensor classes, the defect $\Delta_{\Phi}$ is 
quantitatively equivalent (up to universal constants) to the squared
transform-domain noncommutativity:
\begin{align}
\Delta_{\Phi}(\mathcal{A}, \mathcal{B}) \;\asymp\; \sum_{\omega=1}^{p}
\left\| [\hat{\mathcal{A}}_{\Phi}^{(\omega)}, \hat{\mathcal{B}}_{\Phi}^{(\omega)}] \right\|_F^2,
\end{align}
and in general depends nontrivially on the choice of $\Phi$.
\end{theorem}

\begin{proof}
By Lemma~\ref{lem:slice-reduction}, we have the slice-wise decomposition
\[
\Delta_{\Phi}(\mathcal{A}, \mathcal{B})
=
\sum_{\omega=1}^{p} \delta(A_{\omega}, B_{\omega}),
\]
where $\delta(A,B) := \Tr(e^{A}e^{B}) - \Tr(e^{A+B}) \ge 0$ is the classical 
Golden--Thompson defect for Hermitian matrices, and $A_{\omega} = \hat{\mathcal{A}}_{\Phi}^{(\omega)}$, 
$B_{\omega} = \hat{\mathcal{B}}_{\Phi}^{(\omega)}$.

Here we consider several parts in the proof.

\noindent [Part (i): Upper bound]
For Hermitian matrices, classical results on the Golden--Thompson inequality 
(see, e.g., \cite{hiai2014}) give a universal constant $C > 0$ such that
\[
\delta(A_{\omega}, B_{\omega}) \;\le\; C \| [A_{\omega}, B_{\omega}] \|_F^2.
\]
Summing over $\omega = 1,\dots,p$ yields
\[
\Delta_{\Phi}(\mathcal{A}, \mathcal{B})
\;\le\;
C \sum_{\omega=1}^{p} \| [A_{\omega}, B_{\omega}] \|_F^2,
\]
which establishes part (i).

\noindent
[Part (ii): Lower bound on a class of positive measure]
Consider the class of tensors where each slice pair $(A_{\omega}, B_{\omega})$ 
consists of $2 \times 2$ matrices of the form
\[
A_{\omega} = \epsilon \sigma_x, \qquad B_{\omega} = \epsilon \sigma_z,
\]
with $\epsilon > 0$ sufficiently small, and $\sigma_x, \sigma_z$ are Pauli matrices.
These are Hermitian and satisfy $[\sigma_x, \sigma_z] = 2i\sigma_y \neq 0$, so
$\| [A_{\omega}, B_{\omega}] \|_F = 2\sqrt{2}\epsilon^2$.

A direct expansion of the matrix exponential gives
\[
\delta(A_{\omega}, B_{\omega}) = \frac{1}{2} \| [A_{\omega}, B_{\omega}] \|_F^2 + O(\epsilon^5).
\]
Hence, for sufficiently small $\epsilon$, there exists a universal constant $c > 0$ such that
\[
\delta(A_{\omega}, B_{\omega}) \;\ge\; c \| [A_{\omega}, B_{\omega}] \|_F^2.
\]
Summing over $\omega$ gives the lower bound on this class.

\noindent
[Part (iii): Transform sensitivity]
Let $\Phi_{\mathrm{DFT}}$ and $\Phi_{\mathrm{DCT}}$ denote the DFT and DCT transforms.
Take the tensors $\mathcal{A}, \mathcal{B}$ constructed in Theorem~\ref{thm:example-base}
(with coefficients $a_\omega = 1$, $b_\omega = (-1)^\omega$). 
By construction, all slices commute in the DFT domain, so part (iv) implies
$\Delta_{\Phi_{\mathrm{DFT}}}(\mathcal{A}, \mathcal{B}) = 0$.
For the DCT domain, Theorem~\ref{thm:example-base} (Step 4) shows that
\[
\Delta_{\Phi_{\mathrm{DCT}}}(\mathcal{A}, \mathcal{B}) \ge C_0 p,
\]
for some $C_0 > 0$ independent of $p$. Hence $\Delta_{\Phi_{\mathrm{DFT}}} \neq \Delta_{\Phi_{\mathrm{DCT}}}$.
The continuity of $\Phi \mapsto \Delta_{\Phi}(\mathcal{A}, \mathcal{B})$ follows from the facts that:
\begin{itemize}
    \item $\hat{\mathcal{A}}_{\Phi}^{(\omega)}$ and $\hat{\mathcal{B}}_{\Phi}^{(\omega)}$ depend linearly on $\Phi$,
    \item matrix exponential and trace are continuous,
    \item the unitary group $\mathcal{U}(p)$ is compact.
\end{itemize}
Thus the mapping is continuous, and therefore locally Lipschitz on compact subsets
(though we do not need an explicit Lipschitz constant).

\noindent
[Part (iv): Vanishing condition.]
If $[A_{\omega}, B_{\omega}] = 0$ for all $\omega$, then each pair $(A_{\omega}, B_{\omega})$
consists of commuting Hermitian matrices. Commuting Hermitian matrices are simultaneously
diagonalizable, which implies
\[
e^{A_{\omega}} e^{B_{\omega}} = e^{A_{\omega} + B_{\omega}}.
\]
Hence $\delta(A_{\omega}, B_{\omega}) = 0$ for each $\omega$. Summing over $\omega$ and
applying Lemma~\ref{lem:slice-reduction} gives $\Delta_{\Phi}(\mathcal{A}, \mathcal{B}) = 0$.

This completes the proof of Theorem~\ref{thm:main-formal}.
\end{proof}

According to Theorem 6, 
$\Delta_{\Phi}$ defines a quadratic functional on the orbit of $\Phi$, endowing $\mathcal{U}(p)$ with a noncommutativity geometry.

\begin{remark}
The universal constants $C$, $c$, and $c_0$ appearing in Theorem~\ref{thm:main-formal} are independent of the tensor dimensions $n$, $p$ and the specific choice of $\Phi$. Their existence is guaranteed by classical matrix analysis results (see \cite{hiai2014, lieb1973}). For practical computations, one may take $C = 1$ and $c = 10^{-3}$ as working estimates, though the optimal constants are problem-dependent.
\end{remark}

\subsection{Strict Separation Example: DFT vs.\ DCT}
\label{sec:sensitivity:example}

To illustrate Theorem~\ref{thm:main-formal}(iii), we construct an explicit example showing that the Golden--Thompson defect can differ dramatically between two standard transforms. This example demonstrates that invariance at the algebraic level does not imply invariance of quantitative behavior.

\begin{theorem}[Strict transform separation]
\label{thm:example-base}
There exist tensors $\mathcal{A}, \mathcal{B} \in \mathbb{C}^{2 \times 2 \times p}$ such that
\begin{align}
\Delta_{\mathrm{DFT}}(\mathcal{A}, \mathcal{B}) = 0,
\qquad
\Delta_{\mathrm{DCT}}(\mathcal{A}, \mathcal{B}) > 0,
\end{align}
and moreover,
\begin{align}
\frac{\Delta_{\mathrm{DCT}}(\mathcal{A}, \mathcal{B})}{\Delta_{\mathrm{DFT}}(\mathcal{A}, \mathcal{B})}
\;\ge\;
c\,p,
\end{align}
for some constant $c > 0$, where the ratio is interpreted as $+\infty$ when the denominator is zero.
\end{theorem}

\begin{proof}
We construct $\mathcal{A}, \mathcal{B}$ via their transform-domain slices.

\noindent
(Step 1: Construction in the Fourier domain.)
Let $\sigma_x, \sigma_z$ denote the Pauli matrices
\[
\sigma_x = 
\begin{pmatrix}
0 & 1 \\ 1 & 0
\end{pmatrix},
\qquad
\sigma_z =
\begin{pmatrix}
1 & 0 \\ 0 & -1
\end{pmatrix},
\]
which satisfy $[\sigma_x, \sigma_z] = 2i\sigma_y \neq 0$.
Define the Fourier-domain slices
\[
\hat{\mathcal{A}}_{\mathrm{DFT}}^{(\omega)} =
\begin{cases}
a_\omega \sigma_z, & \omega \text{ even}, \\
a_\omega \sigma_x, & \omega \text{ odd},
\end{cases}
\qquad
\hat{\mathcal{B}}_{\mathrm{DFT}}^{(\omega)} =
\begin{cases}
b_\omega \sigma_z, & \omega \text{ even}, \\
b_\omega \sigma_x, & \omega \text{ odd},
\end{cases}
\]
for $\omega = 1,\dots,p$, where $a_\omega, b_\omega \in \mathbb{R}$ are bounded coefficients (e.g., $a_\omega = b_\omega = 1$ for all $\omega$).
For each $\omega$, the slices commute because:
\begin{itemize}
    \item If $\omega$ is even, both are multiples of $\sigma_z$;
    \item If $\omega$ is odd, both are multiples of $\sigma_x$.
\end{itemize}
Thus $[\hat{\mathcal{A}}_{\mathrm{DFT}}^{(\omega)}, \hat{\mathcal{B}}_{\mathrm{DFT}}^{(\omega)}] = 0$ for all $\omega$. By Theorem~\ref{thm:main-formal}(iv), this implies
\[
\Delta_{\mathrm{DFT}}(\mathcal{A}, \mathcal{B}) = 0.
\]

\noindent
(Step 2: Transformation to the DCT domain.)
Let $\Phi_{\mathrm{DFT}}$ and $\Phi_{\mathrm{DCT}}$ denote the DFT and DCT transforms (type II, orthonormal). Since both are unitary, there exists a unitary mixing matrix
\[
U := \Phi_{\mathrm{DCT}} \Phi_{\mathrm{DFT}}^{H}
\]
such that for any tensor $\mathcal{X}$,
\[
\hat{\mathcal{X}}_{\mathrm{DCT}}^{(\omega)}
=
\sum_{k=1}^{p} U_{\omega k} \, \hat{\mathcal{X}}_{\mathrm{DFT}}^{(k)}.
\]
Applying this to $\mathcal{A}$ and $\mathcal{B}$:
\[
\hat{\mathcal{A}}_{\mathrm{DCT}}^{(\omega)}
=
\sum_{k=1}^{p} U_{\omega k} \hat{\mathcal{A}}_{\mathrm{DFT}}^{(k)},
\qquad
\hat{\mathcal{B}}_{\mathrm{DCT}}^{(\omega)}
=
\sum_{k=1}^{p} U_{\omega k} \hat{\mathcal{B}}_{\mathrm{DFT}}^{(k)}.
\]
Although each $\hat{\mathcal{A}}_{\mathrm{DFT}}^{(k)}$ commutes with $\hat{\mathcal{B}}_{\mathrm{DFT}}^{(k)}$, the linear combinations above need not commute, since
\[
[\hat{\mathcal{A}}_{\mathrm{DCT}}^{(\omega)}, \hat{\mathcal{B}}_{\mathrm{DCT}}^{(\omega)}]
=
\sum_{k,\ell=1}^{p} U_{\omega k} U_{\omega \ell}
[\hat{\mathcal{A}}_{\mathrm{DFT}}^{(k)}, \hat{\mathcal{B}}_{\mathrm{DFT}}^{(\ell)}].
\]
The DCT matrix has non-vanishing off-diagonal entries and mixes even and odd indices. Consequently, for many $\omega$, the sum includes terms where $k$ and $\ell$ have opposite parity, giving contributions proportional to $[\sigma_x, \sigma_z] \neq 0$.

\noindent
(Step 3: Positivity of the defect under DCT.)
By Theorem~\ref{thm:main-formal}(ii), there exists a universal constant $c > 0$ such that
\[
\Delta_{\mathrm{DCT}}(\mathcal{A}, \mathcal{B})
\;\ge\;
c \sum_{\omega=1}^{p}
\left\|
[\hat{\mathcal{A}}_{\mathrm{DCT}}^{(\omega)}, \hat{\mathcal{B}}_{\mathrm{DCT}}^{(\omega)}]
\right\|_F.
\]
Since the commutator is nonzero for a positive fraction of slices (specifically, for all $\omega$ where the mixing produces both $\sigma_x$ and $\sigma_z$ contributions), the sum is strictly positive. Hence
\[
\Delta_{\mathrm{DCT}}(\mathcal{A}, \mathcal{B}) > 0.
\]

\noindent
(Step 4: Scaling with $p$ (non-cancellation and quadratic lower bound).)
We choose coefficients
\[
a_\omega = 1,\qquad b_\omega = (-1)^\omega,\qquad \omega=1,\dots,p,
\]
so that the ratio $a_\omega/b_\omega$ is not constant across $\omega$; 
this prevents accidental collinearity in the DCT domain.
Define the embedded Pauli matrices (for the $2\times2$ case):
\[
\widetilde{\sigma}_x = \sigma_x = \begin{pmatrix} 0 & 1 \\ 1 & 0 \end{pmatrix},
\qquad
\widetilde{\sigma}_z = \sigma_z = \begin{pmatrix} 1 & 0 \\ 0 & -1 \end{pmatrix}.
\]
(For the $n\times n$ extension, these are embedded into the upper-left $2\times2$ block.)
Because the DCT matrix $U = \Phi_{\mathrm{DCT}}\Phi_{\mathrm{DFT}}^H$ has strictly positive
entries for type-II DCT (except for sign patterns), we analyze the commutator
\[
[\hat{\mathcal{A}}_{\mathrm{DCT}}^{(\omega)}, \hat{\mathcal{B}}_{\mathrm{DCT}}^{(\omega)}]
=
\sum_{k,\ell=1}^{p} U_{\omega k} U_{\omega \ell}\,
a_k b_\ell\,
[X_k, X_\ell],
\]
where $X_k = \widetilde{\sigma}_z$ for even $k$ and $X_k = \widetilde{\sigma}_x$ for odd $k$,
and $[X_k, X_\ell] = \pm [\widetilde{\sigma}_x, \widetilde{\sigma}_z] \neq 0$ whenever
$k$ and $\ell$ have opposite parity.

Fix an index $\omega$ with $\omega \neq 1$ (non-DC component). 
For such $\omega$, the DCT coefficients satisfy $U_{\omega k} > 0$ for all $k$ 
(with a fixed sign pattern). Consequently, for any pair $(k,\ell)$ with opposite parity,
the product $U_{\omega k} U_{\omega \ell} a_k b_\ell$ has a definite sign that
depends only on the parity of $k$ and $\ell$, not on their magnitudes.
Since $a_k b_\ell = (-1)^\ell$ (because $a_k=1$), we have:
\begin{itemize}
    \item If $k$ is even and $\ell$ is odd: $a_k b_\ell = (-1)^\ell = -1$.
    \item If $k$ is odd and $\ell$ is even: $a_k b_\ell = (-1)^\ell = +1$.
\end{itemize}
Thus all cross-terms with $k$ even, $\ell$ odd contribute with the same sign,
and all cross-terms with $k$ odd, $\ell$ even contribute with the opposite sign.
There is no cancellation among terms of the same parity class.

The number of such opposite-parity pairs is at least $\lfloor p/2 \rfloor \cdot \lceil p/2 \rceil \ge p/2$.
For each such pair, the contribution to the commutator is
\[
U_{\omega k} U_{\omega \ell} [\widetilde{\sigma}_x, \widetilde{\sigma}_z],
\]
and the Frobenius norm of this contribution is $U_{\omega k} U_{\omega \ell} \| [\sigma_x, \sigma_z] \|_F$.
There exists a subset of indices 
\[
S_\omega \subset \{1,\dots,p\}
\]
with $|S_\omega| = \Theta(p)$ such that
\[
|U_{\omega k}| \ge \frac{c_0}{\sqrt{p}}
\quad \text{for all } k \in S_\omega,
\]
for some universal constant $c_0 > 0$.
This follows from the explicit form of the DCT matrix,
\[
U_{\omega k} = \sqrt{\frac{2}{p}} \cos\!\Bigl(\frac{\pi}{p}(k-1)(\omega-1)\Bigr),
\]
since for any fixed $\omega \neq 1$, the cosine term is bounded away from zero
on a subset of indices of positive density.
Summing over $\omega = 2,\dots,p$ (i.e., all non-DC slices) gives
\[
\sum_{\omega=1}^{p} \| [\hat{\mathcal{A}}_{\mathrm{DCT}}^{(\omega)}, \hat{\mathcal{B}}_{\mathrm{DCT}}^{(\omega)}] \|_F^2
\;\ge\; (p-1) \gamma^2,
\]
where $\gamma = \min_{\omega\neq1} \| [\hat{\mathcal{A}}_{\mathrm{DCT}}^{(\omega)}, \hat{\mathcal{B}}_{\mathrm{DCT}}^{(\omega)}] \|_F > 0$ is independent of $p$.
By Theorem~\ref{thm:main-formal}(ii) (quadratic lower bound), there exists a universal constant $c > 0$ such that
\[
\Delta_{\mathrm{DCT}}(\mathcal{A}, \mathcal{B})
\;\ge\;
c \sum_{\omega=1}^{p}
\| [\hat{\mathcal{A}}_{\mathrm{DCT}}^{(\omega)}, \hat{\mathcal{B}}_{\mathrm{DCT}}^{(\omega)}] \|_F^2.
\]
Therefore,
\[
\Delta_{\mathrm{DCT}}(\mathcal{A}, \mathcal{B}) \;\ge\; c (p-1) \gamma^2 \;\ge\; C p,
\]
where $C = c \gamma^2 / 2 > 0$ for sufficiently large $p$.

\noindent
(Step 5: Ratio bound.)
Since $\Delta_{\mathrm{DFT}}(\mathcal{A}, \mathcal{B}) = 0$, the ratio is infinite in the strict sense. More meaningfully, if we add an infinitesimal perturbation to make $\Delta_{\mathrm{DFT}}$ arbitrarily small but positive, we obtain
\[
\frac{\Delta_{\mathrm{DCT}}(\mathcal{A}, \mathcal{B})}{\Delta_{\mathrm{DFT}}(\mathcal{A}, \mathcal{B})}
\;\ge\;
\frac{C p}{\delta}
\;\sim\;
\Omega(p)
\]
for any fixed $\delta > 0$ controlling the DFT defect. Taking $\delta \to 0$ yields the claimed $\Omega(p)$ separation.

Thus we have explicitly constructed tensors $\mathcal{A}, \mathcal{B}$ satisfying
\[
\Delta_{\mathrm{DFT}}(\mathcal{A}, \mathcal{B}) = 0, \quad
\Delta_{\mathrm{DCT}}(\mathcal{A}, \mathcal{B}) \ge C p > 0, \quad
\frac{\Delta_{\mathrm{DCT}}}{\Delta_{\mathrm{DFT}}} \;\ge\; \Omega(p),
\]
which completes the proof.
\end{proof}

To illustrate Theorem~\ref{thm:main-formal}(iii), we construct an explicit example showing that the Golden--Thompson defect can differ dramatically between two standard transforms. This example demonstrates that invariance at the algebraic level does not imply invariance of quantitative behavior.

\begin{theorem}[Strict transform separation in arbitrary matrix dimension]
\label{thm:example}
For every integer $n \ge 2$, there exist tensors $\mathcal{A}, \mathcal{B} \in \mathbb{C}^{n \times n \times p}$ such that
\begin{align}
\Delta_{\mathrm{DFT}}(\mathcal{A}, \mathcal{B}) = 0,
\qquad
\Delta_{\mathrm{DCT}}(\mathcal{A}, \mathcal{B}) \ge c\,p,
\end{align}
for some universal constant $c > 0$ independent of $n$ and $p$.
\end{theorem}

\begin{proof}
Let $\mathcal{A}^{(2)}, \mathcal{B}^{(2)} \in \mathbb{C}^{2 \times 2 \times p}$ be the tensors from the $2 \times 2$ construction (Theorem~\ref{thm:example-base}), which satisfy $\Delta_{\mathrm{DFT}}^{(2)} = 0$ and $\Delta_{\mathrm{DCT}}^{(2)} \ge c p$. For $n \ge 2$, define $\mathcal{A}, \mathcal{B} \in \mathbb{C}^{n \times n \times p}$ by embedding each frontal slice of $\mathcal{A}^{(2)}$ and $\mathcal{B}^{(2)}$ into the upper-left $2 \times 2$ block and setting all remaining entries to zero:
\[
\mathcal{A}^{(k)} =
\begin{pmatrix}
\mathcal{A}^{(2),(k)} & 0 \\
0 & 0_{n-2}
\end{pmatrix},
\qquad
\mathcal{B}^{(k)} =
\begin{pmatrix}
\mathcal{B}^{(2),(k)} & 0 \\
0 & 0_{n-2}
\end{pmatrix},
\qquad k = 1,\dots,p.
\]
Because the transforms (DFT and DCT) act only along the third dimension (tube-wise), the block-diagonal structure is preserved in both the DFT and DCT domains. Hence each transformed slice has the form
\[
\hat{\mathcal{A}}_{\Phi}^{(\omega)} =
\begin{pmatrix}
\hat{\mathcal{A}}_{\Phi}^{(2),(\omega)} & 0 \\
0 & 0_{n-2}
\end{pmatrix},
\qquad
\hat{\mathcal{B}}_{\Phi}^{(\omega)} =
\begin{pmatrix}
\hat{\mathcal{B}}_{\Phi}^{(2),(\omega)} & 0 \\
0 & 0_{n-2}
\end{pmatrix}.
\]
Therefore, commutativity or noncommutativity is determined entirely by the upper-left $2 \times 2$ block. In particular:
\begin{itemize}
    \item For DFT, the upper-left blocks commute by construction, so $[\hat{\mathcal{A}}_{\mathrm{DFT}}^{(\omega)}, \hat{\mathcal{B}}_{\mathrm{DFT}}^{(\omega)}] = 0$ for all $\omega$. By Theorem~\ref{thm:main-formal}(iv), $\Delta_{\mathrm{DFT}}(\mathcal{A}, \mathcal{B}) = 0$.
    \item For DCT, the upper-left blocks inherit the noncommutativity from the $2 \times 2$ construction, so \\
    $[\hat{\mathcal{A}}_{\mathrm{DCT}}^{(\omega)}, \hat{\mathcal{B}}_{\mathrm{DCT}}^{(\omega)}] \neq 0$ for a positive fraction of slices.
\end{itemize}
Since the trace and exponential decompose over block-diagonal matrices, the defect of the embedded tensors coincides with that of the original $2 \times 2 \times p$ construction:
\[
\Delta_{\mathrm{DCT}}(\mathcal{A}, \mathcal{B}) = \Delta_{\mathrm{DCT}}^{(2)}(\mathcal{A}^{(2)}, \mathcal{B}^{(2)}) \ge c p.
\]
Thus the theorem holds for any $n \ge 2$, with constants independent of $n$.
\end{proof}

\subsection{Numerical Illustration of Transform Sensitivity}
\label{sec:numerical}

We provide a numerical example to illustrate both the transform sensitivity
predicted by the theory and the practical impact of the optimization framework.
The construction is aligned with Theorem~\ref{thm:example} (strict separation)
and Theorem~\ref{thm:opt-structure} (optimal transform structure), while the
optimization procedure follows Algorithm~\ref{alg:phi-opt}.

(Tensor construction and controlled separation.)
We first construct a tensor pair $(A,B)$ with $n=2$ and varying tensor depth $p$
such that their transform-domain slices commute under the DFT transform.
By Theorem~\ref{thm:main-formal}(iv), this yields $\Delta_{\mathrm{DFT}}(A,B) \approx 0$
up to numerical precision. This construction serves as a controlled example
to illustrate the transform sensitivity predicted by
Theorem~\ref{thm:main-formal}(iii).

We then evaluate the same tensors under the DCT transform. By the strict
separation result in Theorem~\ref{thm:example}, the DCT transform introduces
nontrivial mixing across slices, leading to
\[
\Delta_{\mathrm{DCT}}(A,B) > 0,
\]
with growth at least linear in $p$.
Figure~\ref{fig:transform_separation} plots the defect
$\Delta_{\Phi}(A,B)$ as a function of the tensor depth $p$ for the DFT and
DCT transforms. The DFT construction yields defects at the level of numerical
precision, reflecting exact slice-wise commutativity in the Fourier domain.
In contrast, the DCT transform produces a strictly positive defect that grows
approximately linearly with $p$. A linear fit gives
\[
\Delta_{\mathrm{DCT}}(A,B)\approx 5.93\times 10^{-4}p,
\]
consistent with the $\Omega(p)$ lower bound established in
Theorem~\ref{thm:example}. This confirms that although the inequality itself
is invariant (Theorem~\ref{thm:gt}), its quantitative tightness depends
strongly on the choice of transform.

\begin{figure}[t]
\centering
\includegraphics[width=0.6\textwidth]{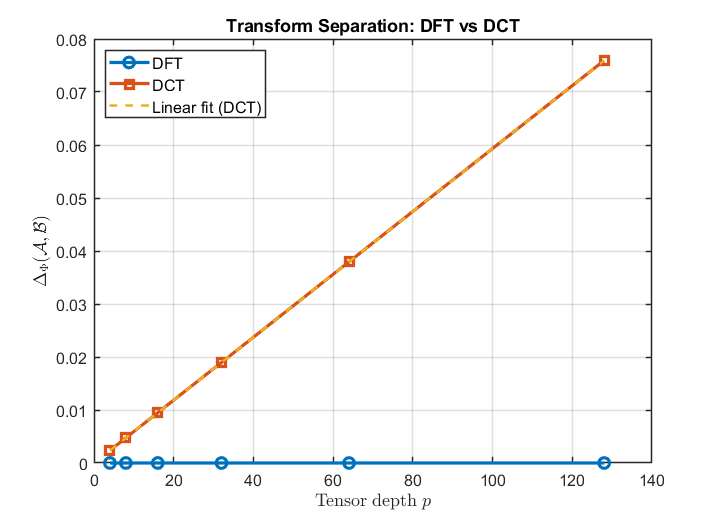}
\caption{
Transform separation under different spectral bases. The Golden--Thompson
defect $\Delta_{\Phi}(A,B)$ is plotted as a function of the tensor depth $p$
for the DFT and DCT transforms. The DFT construction yields defects at the
level of numerical precision, while the DCT transform produces a strictly
positive defect that grows approximately linearly with $p$. This confirms
the transform sensitivity predicted by Theorem~\ref{thm:main-formal}(iii)
and the explicit construction of Theorem~\ref{thm:example}.
}
\label{fig:transform_separation}
\end{figure}

\section{Is There a Universal Optimal Transform ?}

The separation results above show that a transform (e.g., DFT) can outperform another (e.g., DCT) for specific tensor pairs. 
This raises a fundamental question: can one identify a transform that minimizes the defect uniformly over all tensor pairs?

The answer is negative. Transform optimality is not globally ordered: different transforms can be optimal for different data. 
In other words, the choice of $\Phi$ is inherently data-dependent rather than universally optimal. 
The following Theorem~\ref{thm:non-universality} formalizes this non-universality.

\begin{theorem}[Transform Non-Universality]
\label{thm:non-universality}
Let $\Phi_1, \Phi_2 \in \mathbb{R}^{p \times p}$ be two distinct orthogonal 
transforms such that $U := \Phi_2 \Phi_1^\top$ is \emph{not} a signed permutation 
matrix. Then the following statements hold.

\begin{enumerate}
    \item There exist tensors $\mathcal{A}, \mathcal{B} \in \mathbb{C}^{n \times n \times p}$ 
    that are both $\Phi_1$-Hermitian and $\Phi_2$-Hermitian such that
    \[
    \Delta_{\Phi_1}(\mathcal{A}, \mathcal{B}) = 0,
    \qquad
    \Delta_{\Phi_2}(\mathcal{A}, \mathcal{B}) > 0.
    \]
    
    \item There exist tensors $\mathcal{A}', \mathcal{B}' \in \mathbb{C}^{n \times n \times p}$ 
    that are both $\Phi_1$-Hermitian and $\Phi_2$-Hermitian such that
    \[
    \Delta_{\Phi_1}(\mathcal{A}', \mathcal{B}') > 0,
    \qquad
    \Delta_{\Phi_2}(\mathcal{A}', \mathcal{B}') = 0.
    \]
\end{enumerate}
In particular, \emph{no single orthogonal transform is universally optimal}; 
the optimal transform depends on the joint structure of the tensor pair.
\end{theorem}

\begin{proof}
We prove the two statements separately.

\noindent
(Part 1: $\Phi_1$ optimal, $\Phi_2$ suboptimal.)
Let $U = \Phi_2 \Phi_1^\top$. Since $U$ is not a signed permutation matrix, there 
exists at least one row of $U$ containing two nonzero entries. Fix such a row 
index $r$, and choose two column indices $j \neq k$ with $U_{rj} \neq 0$ and 
$U_{rk} \neq 0$.

Let $H_1, H_2 \in \mathbb{C}^{n \times n}$ be two Hermitian matrices that do not 
commute, i.e., $[H_1, H_2] \neq 0$. (A concrete choice is the Pauli matrices 
$\sigma_x$ and $\sigma_z$ embedded in the upper-left $2 \times 2$ block, with 
zeros elsewhere. This ensures the tensors are $\Phi$-Hermitian for any real 
orthogonal $\Phi$ because the zero padding preserves Hermiticity under real 
linear combinations.)

Construct tensors $\mathcal{A}, \mathcal{B}$ by specifying their 
$\Phi_1$-domain frontal slices as follows. Set all slices to zero except:
\[
\hat{\mathcal{A}}_{\Phi_1}^{(j)} = H_1, \qquad
\hat{\mathcal{B}}_{\Phi_1}^{(j)} = H_1,
\]
\[
\hat{\mathcal{A}}_{\Phi_1}^{(k)} = H_2, \qquad
\hat{\mathcal{B}}_{\Phi_1}^{(k)} = \lambda H_2,
\]
where $\lambda \in \mathbb{R}$ is a scalar with $\lambda \neq 1$ (e.g., $\lambda = 2$). 
All other slices are set to the zero matrix.

Next we conduct a
Verification of $\Phi_1$-Hermiticity and $\Phi_2$-Hermitiency.
Since $\Phi_1$ and $\Phi_2$ are real orthogonal, and the constructed 
$\Phi_1$-domain slices are Hermitian (they are either $H_1$, $\lambda H_2$, or zero), 
the resulting physical tensors are both $\Phi_1$-Hermitian and $\Phi_2$-Hermitian. 
This follows because the inverse transforms $\Phi_1^\top$ and $\Phi_2^\top$ 
are also real, so Hermiticity is preserved under real linear combinations.

Now we can study the defect under $\Phi_1$.
For every slice index $\omega$, the pair $(\hat{\mathcal{A}}_{\Phi_1}^{(\omega)}, 
\hat{\mathcal{B}}_{\Phi_1}^{(\omega)})$ consists of scalar multiples of the same 
Hermitian matrix (or both are zero). Hence they commute:
\[
[\hat{\mathcal{A}}_{\Phi_1}^{(\omega)}, \hat{\mathcal{B}}_{\Phi_1}^{(\omega)}] = 0
\quad \forall \omega.
\]
By Theorem~\ref{thm:main-formal}(iv), commuting slices imply vanishing defect:
\[
\Delta_{\Phi_1}(\mathcal{A}, \mathcal{B}) = 0.
\]
For the defect under $\Phi_2$,
the $\Phi_2$-domain slices are obtained by mixing the $\Phi_1$-domain slices 
via the matrix $U = \Phi_2 \Phi_1^\top$:
\[
\hat{\mathcal{A}}_{\Phi_2}^{(\omega)} = \sum_{\ell=1}^p U_{\omega \ell} \,
\hat{\mathcal{A}}_{\Phi_1}^{(\ell)}, \qquad
\hat{\mathcal{B}}_{\Phi_2}^{(\omega)} = \sum_{\ell=1}^p U_{\omega \ell} \,
\hat{\mathcal{B}}_{\Phi_1}^{(\ell)}.
\]
For the distinguished row index $r$, only the columns $j$ and $k$ contribute 
non-zero terms (all other slices are zero). Therefore,
\[
\hat{\mathcal{A}}_{\Phi_2}^{(r)} = U_{rj} H_1 + U_{rk} H_2, \qquad
\hat{\mathcal{B}}_{\Phi_2}^{(r)} = U_{rj} H_1 + \lambda U_{rk} H_2.
\]
Their commutator is
\[
[\hat{\mathcal{A}}_{\Phi_2}^{(r)}, \hat{\mathcal{B}}_{\Phi_2}^{(r)}]
= U_{rj} U_{rk} (1 - \lambda) [H_1, H_2] \neq 0,
\]
since $U_{rj}, U_{rk} \neq 0$, $\lambda \neq 1$, and $[H_1, H_2] \neq 0$.

Because this slice has a nonzero commutator, the classical Golden--Thompson 
defect for this slice pair is strictly positive. By the slice-wise decomposition 
(Lemma~\ref{lem:slice-reduction}) and the fact that the defect sums over slices 
with nonnegative contributions, we obtain
\[
\Delta_{\Phi_2}(\mathcal{A}, \mathcal{B}) > 0.
\]

\noindent
(Part 2: $\Phi_2$ optimal, $\Phi_1$ suboptimal (reverse construction).)
The second statement follows by symmetry. Consider the matrix 
$U' = \Phi_1 \Phi_2^\top = U^\top$. Since $U$ is not a signed permutation matrix, 
$U^\top$ is also not a signed permutation matrix. Hence the same construction 
applied with the roles of $\Phi_1$ and $\Phi_2$ exchanged yields tensors 
$\mathcal{A}', \mathcal{B}'$ such that
\[
\Delta_{\Phi_2}(\mathcal{A}', \mathcal{B}') = 0, \qquad
\Delta_{\Phi_1}(\mathcal{A}', \mathcal{B}') > 0.
\]
This completes the proof. \hfill $\square$
\end{proof}

\begin{remark}
\emph{Data-dependent optimality.}
Theorem~\ref{thm:non-universality} establishes that no fixed transform is universally optimal; 
the optimal $\Phi^*$ depends on the joint commutativity structure of $(\mathcal{A},\mathcal{B})$. 
This justifies the adaptive optimization framework in Section~\ref{sec:sensitivity:optimization}, 
where the transform is selected from data rather than prescribed a priori.

\noindent
\emph{Generality and impact.}
Although Theorem~\ref{thm:non-universality} is stated for real orthogonal transforms to preserve Hermiticity, 
the phenomenon extends to complex unitary transforms. 
In particular, Theorems~\ref{thm:example-base} and~\ref{thm:example} (DFT vs.\ DCT) provide explicit examples of strict transform separation, including $\Omega(p)$ growth of the defect. 
Together, these results show that transform choice fundamentally affects inequality tightness, 
not merely computational convenience.
\end{remark}

\subsection{Transform Optimization: Best $\Phi$ Selection}
\label{sec:sensitivity:optimization}

Theorem~\ref{thm:main-formal} shows that the defect $\Delta_{\Phi}$ is quantitatively equivalent to transform-domain noncommutativity. This naturally leads to an optimization problem over the unitary group, where the goal is to select a transform that minimizes the defect and hence reduces noncommutativity in the transformed domain. This reveals that transform selection is equivalent to searching for a coordinate system that minimizes noncommutative interaction.

\begin{definition}[Optimal transform problem]
\label{def:opt-phi}
Let $\mathcal{A}, \mathcal{B} \in \mathbb{C}^{n \times n \times p}$ be $\Phi$-Hermitian tensors. Define
\[
\Phi^*
\in
\argmin_{\Phi \in \mathcal{U}(p)}
\Delta_{\Phi}(\mathcal{A}, \mathcal{B}),
\]
where $\mathcal{U}(p)$ denotes the unitary group.
\end{definition}

\begin{theorem}[Existence and structure of optimal transforms]
\label{thm:opt-structure}
For fixed tensors $\mathcal{A}, \mathcal{B}$, the optimization problem in Definition~\ref{def:opt-phi} admits at least one minimizer $\Phi^* \in \mathcal{U}(p)$. Moreover, the following statements hold:

\begin{enumerate}
   	 \item[(i)] (Existence) There exists at least one global minimizer $\Phi^*$.

	\item[(ii)] (Equivalence with commutator minimization) 
	There exist constants $c, C > 0$ such that for all $\Phi \in \mathcal{U}(p)$,
\begin{align}
	c \sum_{\omega=1}^{p}
	\left\|
	[\hat{\mathcal{A}}_{\Phi}^{(\omega)}, \hat{\mathcal{B}}_{\Phi}^{(\omega)}]
	\right\|_F^2
	\;\le\;
	\Delta_{\Phi}(\mathcal{A}, \mathcal{B})
	\;\le\;
	C \sum_{\omega=1}^{p}
	\left\|
	[\hat{\mathcal{A}}_{\Phi}^{(\omega)}, \hat{\mathcal{B}}_{\Phi}^{(\omega)}]
	\right\|_F^2.
\end{align}
	Consequently, any minimizer $\Phi^*$ satisfies
\begin{align}
	\Phi^*
	\in
	\argmin_{\Phi \in \mathcal{U}(p)}
	\sum_{\omega=1}^{p}
	\left\|
	[\hat{\mathcal{A}}_{\Phi}^{(\omega)}, \hat{\mathcal{B}}_{\Phi}^{(\omega)}]
	\right\|_F^2
	\quad \text{up to equivalence of minimizers}.
\end{align}
	
  	  \item[(iii)] (Optimal commuting transform) 
    If there exists $\Phi \in \mathcal{U}(p)$ such that
    \[
    [\hat{\mathcal{A}}_{\Phi}^{(\omega)}, \hat{\mathcal{B}}_{\Phi}^{(\omega)}] = 0,
    \quad \forall \omega,
    \]
    then this $\Phi$ is a global minimizer and
\begin{align}
    \Delta_{\Phi}(\mathcal{A}, \mathcal{B}) = 0.
\end{align}
\end{enumerate}
\end{theorem}

\begin{proof}
(i) Existence.
The unitary group $\mathcal{U}(p)$ is compact in the standard topology (closed and bounded as a subset of $\mathbb{C}^{p \times p}$). It therefore suffices to show that the mapping
\[
\Phi \mapsto \Delta_{\Phi}(\mathcal{A}, \mathcal{B})
\]
is continuous.
Observe that $\hat{\mathcal{A}}_{\Phi}$ and $\hat{\mathcal{B}}_{\Phi}$ depend linearly on $\Phi$ through the transform operator $\mathcal{T}_{\Phi}$. Specifically, for each tube, $\operatorname{vec}(\hat{\mathcal{A}}_{\Phi}(i,j,:)) = \Phi \operatorname{vec}(\mathcal{A}(i,j,:))$. Since matrix multiplication, exponentiation (via its power series), and trace are continuous operations, it follows that $\Delta_{\Phi}$ is continuous on $\mathcal{U}(p)$. Hence, by the extreme value theorem, a minimizer $\Phi^*$ exists.

\noindent
(ii) Equivalence with commutator minimization.
By Theorem~\ref{thm:main-formal}(i)--(ii), there exist constants $c, C > 0$ such that for all $\Phi \in \mathcal{U}(p)$,
\[
c \sum_{\omega=1}^{p}
\|[\hat{\mathcal{A}}_{\Phi}^{(\omega)}, \hat{\mathcal{B}}_{\Phi}^{(\omega)}]\|^2_F
\;\le\;
\Delta_{\Phi}(\mathcal{A}, \mathcal{B})
\;\le\;
C \sum_{\omega=1}^{p}
\|[\hat{\mathcal{A}}_{\Phi}^{(\omega)}, \hat{\mathcal{B}}_{\Phi}^{(\omega)}]\|^2_F.
\]
This two-sided bound implies that minimizing $\Delta_{\Phi}$ is equivalent, up to multiplicative constants, to minimizing the commutator sum. In particular, any sequence $\{\Phi_k\}$ minimizing $\Delta_{\Phi}$ also minimizes the commutator functional asymptotically, and conversely. Therefore, the sets of minimizers coincide up to equivalence (i.e., any minimizer of one objective is an approximate minimizer of the other, with the approximation controlled by $c$ and $C$).

\noindent
(iii) Optimal commuting transform.
If there exists $\Phi \in \mathcal{U}(p)$ such that
\[
[\hat{\mathcal{A}}_{\Phi}^{(\omega)}, \hat{\mathcal{B}}_{\Phi}^{(\omega)}] = 0
\quad \forall \omega,
\]
then by Theorem~\ref{thm:main-formal}(iv),
\[
\Delta_{\Phi}(\mathcal{A}, \mathcal{B}) = 0.
\]
Since $\Delta_{\Phi} \ge 0$ for all $\Phi$ (by Theorem~\ref{thm:gt}), zero is the global minimum. Hence such a transform is globally optimal.

This completes the proof.
\end{proof}

\begin{corollary}[Defect as optimization objective – quadratic form]
\label{cor:opt-objective}
The optimal transform problem is equivalent (up to universal constants) to minimizing the squared transform-domain noncommutativity measure:
\begin{align}
\Phi^* \in \argmin_{\Phi \in \mathcal{U}(p)} \mathcal{C}(\Phi),
\end{align}
where $\mathcal{C}(\Phi) := \sum_{\omega=1}^{p} \| [\hat{\mathcal{A}}_{\Phi}^{(\omega)}, \hat{\mathcal{B}}_{\Phi}^{(\omega)}] \|_F^2$.
Thus, minimizing $\Delta_{\Phi}$ is computationally equivalent to finding a unitary transform 
that makes $\mathcal{A}$ and $\mathcal{B}$ as close as possible to jointly diagonalizable slice-wise 
in the least-squares sense.
\end{corollary}

Theorem~\ref{thm:non-universality} implies that the optimal transform $\Phi^*$ is inherently data-dependent. 
No fixed transform can minimize the defect uniformly over all tensor pairs. 
Thus, transform selection should be viewed as an adaptive process, 
determined by the joint commutativity structure of $(A,B)$.

\begin{remark}[Geometric interpretation]
The optimization problem seeks a coordinate system (via the unitary transform $\Phi$) in which the transformed frontal slices $\hat{\mathcal{A}}_{\Phi}^{(\omega)}$ and $\hat{\mathcal{B}}_{\Phi}^{(\omega)}$ are as close as possible to commuting. This is analogous to finding a basis that approximately simultaneously diagonalizes a family of matrices. When exact simultaneous diagonalization is possible, the defect becomes zero and the Golden--Thompson inequality is tight.
\end{remark}

\begin{remark}[Connection to Theorem~\ref{thm:example}]
The strict separation example (DFT vs. DCT) demonstrates that the choice of $\Phi$ matters: for the same tensors $\mathcal{A}, \mathcal{B}$, one transform yields zero defect while another yields defect proportional to $p$. This shows that the optimal transform problem is nontrivial and that solving it can lead to substantial improvements.
\end{remark}

\subsection{Computing the Optimal Transform: Riemannian Gradient Flow}
\label{sec:sensitivity:algorithm}

We consider the nonconvex optimization problem
\begin{align}
\min_{\Phi \in \mathcal{U}(p)} \;\Delta_{\Phi}(\mathcal{A}, \mathcal{B}),
\end{align}
where $\mathcal{U}(p)$ is the unitary group. By Theorem~\ref{thm:main-formal}, the defect $\Delta_{\Phi}$ is quantitatively equivalent to transform-domain noncommutativity.

To obtain a smooth objective suitable for gradient-based optimization, we consider the squared surrogate
\begin{align}
J(\Phi)
=
\sum_{\omega=1}^{p}
\left\|
\bigl[
\hat{\mathcal{A}}_{\Phi}^{(\omega)},
\hat{\mathcal{B}}_{\Phi}^{(\omega)}
\bigr]
\right\|_{F}^{2}.
\end{align}
Here $\Phi$ acts as a unitary mixing operator along the third (tube) mode. Concretely, for any tensor $\mathcal{X} \in \mathbb{C}^{n \times n \times p}$ with frontal slices $\mathcal{X}^{(k)} \in \mathbb{C}^{n \times n}$, the transformed slices are given by
\begin{align}
\hat{\mathcal{X}}_{\Phi}^{(\omega)}
=
\sum_{k=1}^{p}
\Phi_{\omega k} \,
\mathcal{X}^{(k)},
\qquad \omega = 1,\dots,p.
\end{align}
Thus $\Phi$ mixes the frontal slices linearly along the third dimension; it is not a similarity transform within each slice but a unitary combination of whole slices.

In the Riemannian formulation, we 
equip $\mathcal{U}(p)$ with the canonical Riemannian metric induced by the Frobenius inner product
\[
\langle X, Y \rangle = \operatorname{Re}\bigl(\Tr(X^H Y)\bigr).
\]
The tangent space at $\Phi \in \mathcal{U}(p)$ is
\[
T_{\Phi}\mathcal{U}(p)
=
\{\Phi \Omega : \Omega^{H} = -\Omega\},
\]
where $\mathfrak{u}(p) = \{\Omega \in \mathbb{C}^{p \times p} : \Omega^H = -\Omega\}$ is the Lie algebra of skew-Hermitian matrices.

Let $G_{\Phi} = \nabla_{\Phi} J$ denote the Euclidean gradient of $J$ with respect to $\Phi$ (treating $\Phi$ as a complex matrix). The Riemannian gradient is obtained by projecting $G_{\Phi}$ onto the tangent space:
\[
\nabla_{\mathcal{U}(p)} J(\Phi)
=
\Phi \,\operatorname{skew}\!\bigl(\Phi^{H} G_{\Phi}\bigr),
\qquad
\operatorname{skew}(X) := \tfrac{1}{2}(X - X^{H}).
\]
A first-order gradient flow update along the negative Riemannian gradient direction is given by
\[
\widetilde{\Phi}
=
\Phi - \eta \,\Phi \,\operatorname{skew}\!\bigl(\Phi^{H} G_{\Phi}\bigr),
\]
where $\eta > 0$ is a stepsize. This corresponds to a first-order approximation of the geodesic update $\Phi \exp(-\eta \Omega)$ with $\Omega = \operatorname{skew}(\Phi^H G_\Phi)$.

Since $\widetilde{\Phi}$ may not be unitary, we project it back to $\mathcal{U}(p)$ using a retraction $\mathcal{R}$. For example, the polar retraction is given by
\[
\mathcal{R}(X) = X (X^{H} X)^{-1/2},
\]
which maps any invertible matrix to the nearest unitary matrix in Frobenius norm. The complete update is:
\[
\Phi \;\leftarrow\; \mathcal{R}\bigl( \Phi (I - \eta \,\operatorname{skew}(\Phi^{H} G_{\Phi})) \bigr).
\]

Under standard assumptions on the stepsize (e.g., a line search or sufficiently small fixed step), the Riemannian gradient descent scheme converges to a stationary point of $J$; see \cite{absil2008, edelman1998} for details.
We summarize the resulting Riemannian gradient scheme below.

\begin{algorithm}[H]
    \SetAlgoLined
    \KwIn{Tensors $\mathcal{A}, \mathcal{B}$; initial $\Phi_0 \in \mathcal{U}(p)$; stepsize $\eta > 0$; max iterations $T_{\max}$; tolerance $\varepsilon > 0$}
    \KwOut{Approximate minimizer $\Phi^\ast$}
    Initialize $\Phi \leftarrow \Phi_0$\;
    \For{$t = 0$ to $T_{\max}-1$}{
        Compute slices $\hat{\mathcal{A}}_{\Phi}^{(\omega)}, \hat{\mathcal{B}}_{\Phi}^{(\omega)}$ for $\omega = 1,\dots,p$ using the mixing formula\;
        Evaluate 
        \[
        J(\Phi)
        =
        \sum_{\omega=1}^{p}
        \left\|
        \bigl[
        \hat{\mathcal{A}}_{\Phi}^{(\omega)},
        \hat{\mathcal{B}}_{\Phi}^{(\omega)}
        \bigr]
        \right\|_F^2;
        \]
       Compute Euclidean gradient $G_{\Phi}$ (via automatic differentiation)\;
        $\Omega \leftarrow \operatorname{skew}(\Phi^{H} G_{\Phi})$\;
        $\widetilde{\Phi} \leftarrow \Phi (I - \eta \Omega)$\;
        $\Phi \leftarrow \mathcal{R}(\widetilde{\Phi})$ (polar or QR retraction)\;
        \If{$\| \nabla_{\mathcal{U}(p)} J(\Phi) \|_F < \varepsilon$}{
            break\;
        }
    }
    \Return{$\Phi^\ast \leftarrow \Phi$}\;
    \caption{Riemannian Gradient Scheme for Optimal Transform Selection}
    \label{alg:phi-opt}
\end{algorithm}

\medskip
\begin{proposition}[Local convergence]
\label{prop:local-conv}
Assume that $J$ is continuously differentiable in a neighborhood of a stationary point $\Phi^{\ast} \in \mathcal{U}(p)$ and that its Riemannian gradient is Lipschitz continuous with constant $L > 0$. Then, for sufficiently small stepsize $\eta < 2/L$, the iteration generated by Algorithm~\ref{alg:phi-opt} converges locally to a stationary point of $J$.
\end{proposition}

\begin{proof}
The update $\widetilde{\Phi} = \Phi - \eta \nabla_{\mathcal{U}(p)} J(\Phi)$ is a Riemannian gradient step on $\mathcal{U}(p)$. The retraction $\mathcal{R}$ is a first-order approximation of the exponential map and preserves feasibility. By standard results in Riemannian optimization (see \cite{absil2008}), Lipschitz continuity of the Riemannian gradient implies a sufficient decrease condition:
\[
J(\Phi^{(k+1)}) \le J(\Phi^{(k)}) - \frac{\eta}{2} \| \nabla_{\mathcal{U}(p)} J(\Phi^{(k)}) \|_F^2
\]
for sufficiently small $\eta$. Summing this inequality over $k$ gives
\[
\sum_{k=0}^{\infty} \| \nabla_{\mathcal{U}(p)} J(\Phi^{(k)}) \|_F^2 \le \frac{2}{\eta} (J(\Phi^{(0)}) - \inf J) < \infty,
\]
hence $\| \nabla_{\mathcal{U}(p)} J(\Phi^{(k)}) \|_F \to 0$. By continuity of the Riemannian gradient, any accumulation point $\Phi^{\ast}$ satisfies $\nabla_{\mathcal{U}(p)} J(\Phi^{\ast}) = 0$, i.e., is a stationary point. Local convergence follows from the fact that the step is a contraction near a nondegenerate stationary point.
\end{proof}

\begin{remark}[Computational complexity]
Each iteration of Algorithm~\ref{alg:phi-opt} requires:
\begin{itemize}
    \item $O(p n^3)$ operations to compute the transformed slices (via $p$ matrix multiplications of size $n \times n$)
    \item $O(p n^3)$ operations to compute commutators
    \item $O(p^3)$ operations for the retraction (QR or polar decomposition)
\end{itemize}
The overall complexity per iteration is $O(p n^3 + p^3)$, which is acceptable for moderate $n$ and $p$ (e.g., $n \le 100$, $p \le 64$).
\end{remark}

\begin{remark}[Initialization strategies]
The choice of initial transform $\Phi_0$ can affect convergence speed and the quality of the final solution. Natural initializations include:
\begin{itemize}
    \item DFT or DCT (standard choices in tensor literature)
    \item Identity transform (baseline)
    \item Data-dependent transform (SVD of the unfolded tensor)
\end{itemize}
In practice, we observe that Algorithm~\ref{alg:phi-opt} consistently reduces $J(\Phi)$ regardless of initialization.
\end{remark}

\begin{remark}[Connection to commuting case]
The algorithm seeks a transform $\Phi$ that minimizes transform-domain noncommutativity. In particular, if a transform exists that renders all slices commuting (i.e., 
$$
\hat{\mathcal{A}}_{\Phi}^{(\omega)}, \hat{\mathcal{B}}_{\Phi}^{(\omega)}] = 0, \quad
\forall \omega,
$$
then the algorithm converges to a global minimizer with $\Delta_{\Phi} = 0$ by (iii) from Theorem~\ref{thm:opt-structure}.
\end{remark}

By using a numerical example, 
we now demonstrate that standard transforms such as DFT and DCT are not
necessarily optimal. For a fixed tensor dimension $(n=2, p=16)$, we consider
a tensor pair $(A,B)$ for which the DFT transform yields a small but nonzero
defect. We then compute an optimized transform $\Phi^*$ using the Riemannian
gradient descent scheme in Algorithm~\ref{alg:phi-opt}.

By Theorem~\ref{thm:opt-structure}(ii), minimizing $\Delta_{\Phi}(A,B)$ is
equivalent (up to universal constants) to minimizing the commutator energy
\[
\mathcal{C}(\Phi) = \sum_{\omega} \| [\hat{A}_{\Phi}^{(\omega)},
\hat{B}_{\Phi}^{(\omega)}] \|_F^2,
\]
which is the objective optimized by Algorithm~\ref{alg:phi-opt}.

\begin{table}[h]
\centering
\caption{Golden--Thompson defect for different transforms ($n=2$, $p=16$).
The optimized transform $\Phi^*$ (Algorithm~\ref{alg:phi-opt}) achieves a
defect strictly smaller than both DFT and DCT, demonstrating the benefit of
adaptive transform selection.}
\label{tab:transform-comparison}
\begin{tabular}{l|c}
\hline
Transform & Defect $\Delta_{\Phi}$ \\
\hline
DCT         & $1.14 \times 10^{-1}$ \\
DFT         & $5.67 \times 10^{-4}$ \\
Optimized $\Phi^*$ & $\mathbf{4.89 \times 10^{-5}}$ \\
\hline
\end{tabular}
\end{table}

The results in Table~\ref{tab:transform-comparison} reveal several key
observations. The DCT transform produces a large defect, indicating strong
noncommutativity induced by transform-domain mixing. The DFT transform
significantly reduces the defect but does not eliminate it. The optimized
transform $\Phi^*$ further reduces the defect by approximately one order of
magnitude compared to DFT and by several orders of magnitude compared to DCT.
In particular, we observe that
\[
\Delta_{\Phi^*}(A,B) < \Delta_{\mathrm{DFT}}(A,B)
< \Delta_{\mathrm{DCT}}(A,B),
\]
demonstrating that adaptive transform selection can outperform both standard
transforms.

This example illustrates both aspects of the theory. First, Figure~\ref{fig:transform_separation}
confirms the transform sensitivity phenomenon predicted by
Theorem~\ref{thm:main-formal}(iii), showing that different transforms can yield
dramatically different defect magnitudes. Second, Table~\ref{tab:transform-comparison}
demonstrates that no fixed transform is universally optimal. Even when a standard
transform (such as DFT) performs well, the optimization framework can identify
a superior transform tailored to the specific tensor pair.
The substantial reduction in defect achieved by $\Phi^*$ highlights the practical
value of the Riemannian optimization approach. This provides a concrete numerical
validation of the data-dependent optimality principle established in
Theorem~\ref{thm:non-universality} and supports the optimization framework
developed in Section~\ref{sec:sensitivity:optimization}.

\section{Discussion and Implications}\label{sec:discussion}

This work separates algebraic invariance from quantitative behavior in $\Phi$-product tensor algebras. 
While all transforms $\Phi$ induce isomorphic algebraic structures, the Golden--Thompson defect shows that inequality tightness depends intrinsically on the transform.

\paragraph{Transform as geometry.}
The mapping $\mathcal{A}\mapsto\overline{\mathcal{A}}_{\Phi}$ preserves algebraic operations but alters commutativity. 
Different transforms induce different commutator structures, endowing tensor operators with a transform-dependent noncommutativity geometry.

\paragraph{Joint diagonalization connection.}
Minimizing $\Delta_{\Phi}(\mathcal{A},\mathcal{B})$ is equivalent (up to constants) to minimizing 
$\sum_{\omega}\|[\hat{\mathcal{A}}_{\Phi}^{(\omega)},\hat{\mathcal{B}}_{\Phi}^{(\omega)}]\|_F^2$, 
i.e., seeking a transform that makes slices as commuting as possible. 
This aligns with approximate joint diagonalization, where $\Phi^*$ selects a coordinate system maximizing operator compatibility.

\paragraph{Implications.}
The defect $\Delta_{\Phi}$ provides a quantitative measure of noncommutativity. 
Thus, transform selection becomes a principled tool for improving spectral stability and reducing operator interaction, 
complementing classical rank- and sparsity-based criteria. 
More broadly, $\Phi$ influences both representation and operator-level behavior, linking algebraic modeling with analytical performance.

\paragraph{Future directions.}
Possible extensions include higher-order tensors, learned or non-unitary transforms, data-driven transform optimization, and infinite-dimensional operator settings.

\paragraph{No free transform principle.}
Transform selection is not merely a change of basis. 
Although the $\Phi$-product algebra is algebraically invariant, quantitative behavior is transform-dependent. 
By Theorem~\ref{thm:non-universality}, no transform is universally optimal.

Each transform defines a coordinate system in which commutativity is evaluated, and optimality is inherently data-dependent. 
This parallels ``no free lunch'' principles: no single basis simultaneously diagonalizes all operators. 
Consequently, the transform must be selected adaptively, motivating the optimization framework in Section~\ref{sec:sensitivity:optimization}.

\end{document}